\documentclass[11pt]{amsart}
\usepackage{latexsym,amssymb,amsmath}
\usepackage{amscd}
\usepackage[mathscr]{eucal}
\usepackage{verbatim}
\usepackage{epsfig}
\def\intprod{\negthinspace
\mathbin{\raisebox{.4ex}{\hbox{\vrule height .5pt width 5pt depth 0pt %
        \vrule height 3pt width .5pt depth 0pt}}}}

\def\ctimes{\times \cdots\times}

\def\t{\tau}

\def\trank{\text{rank}}

\def\BF{\mathbb F}

\def\BP{\mathbb P}
\def\pp#1{\mathbb P^{#1}}

\def\pp#1{{\mathbb P}^{#1}}
\def\tdim{\rm dim}
\def\hd{,...,}
\def\ww{\wedge}
\def\upperp{{}^\perp}

\def\inv{{}^{-1}}

\def\cB{{\mathcal B}}

\def\cE{{\mathcal E}}

\def\cC{{\mathcal C}}

\def\cS{{\mathcal S}}\def\cM{{\mathcal M}}

\def\cO{{\mathcal O}}

\def\BG{\mathbb G}
\def\11{\mathbf 1}

\def\l{\lambda}
\def\a{\alpha}

\def\s{\sigma}

\def\d{\delta}

\def\ot{{\mathord{\,\otimes }\,}}
\def\op{{\mathord{\,\oplus }\,}}

\def\ctimes{{\mathord{\times\cdots\times}\;}}

\def\ra{{\mathord{\;\rightarrow\;}}}

\def\tim{\text{Image}\,}
\def\tdeg{\text{deg} }

\def\tdim{\text{dim}\,}
\def\tcodim{\text{codim}\,}

\def\tker{\text{ker}\,}

\def\trank{\text{rank}\,}

\def\tlim{\text{lim}}

\def\be{\begin{equation}}
\def\ene{\end{equation}}

\newtheorem{theorem}{Theorem}[section]
\newtheorem{proposition}[theorem]{Proposition}
\newtheorem{lemma}[theorem]{Lemma}

\newtheorem{conjecture}[theorem]{Conjecture}

\theoremstyle{definition}

\theoremstyle{remark}

\def\ci{\mathcal I}
\def\sxe{\s_{(X,E)}}

\def\bpwx{{\BP (W/\hat x) }}

\def\bcx{\bold C_x}
\def\tmi{q^*({\bold M_j})^*\ot \cO_{\BP(\bold S)}(\d_j)}
\def\tmii#1{q^*({\bold M_{#1}})^*\ot \cO_{\BP(\bold S)}(\d_{#1})}
\def\htmi{{q^*(\hat{\bold M}_j)^*\ot \cO_{\BP(\bold S)}(\d_j)}}

\DeclareMathOperator{\ttop}{{top}}
\DeclareMathOperator{\Damp}{Damp}
\def\co{\colon\thinspace}

\newcommand{\sing}[1]{{#1}_{\text{sing}}}

\begin{document}
\title[On the Debarre-de Jong and Beheshti-Starr conjectures]{On the Debarre-de Jong and Beheshti-Starr conjectures on hypersurfaces with too many lines}
\author{J.M. Landsberg and Orsola Tommasi}

\begin{abstract}  
We show that the Debarre-de Jong conjecture that the Fano scheme of lines on a smooth hypersurface of   degree  at most $n$ in $\pp n$ must
have its expected dimension,
 and the Beheshti-Starr conjecture  that bounds the dimension of the Fano scheme of lines  for hypersurfaces of degree at least $n$ in $\pp n$,  reduce to determining if the intersection of the 
top Chern classes of certain vector bundles 
is nonzero.
\end{abstract}

 \thanks{Landsberg supported by NSF grant DMS-0805782}
\email{jml@math.tamu.edu, tommasi@math.uni-hannover.de }
\maketitle



\section{Introduction}

Let $K$ be an algebraically closed field of characteristic zero.
Write $\sing X$ for the singular points of a variety $X$, $\pp n=K\pp n$ and $\BG(\pp 1,\pp n)=G(2,n+1)$ for the Grassmannian.

\begin{conjecture}\label{mainthm} Let $X^{n-1}\subset  \pp n$ be a
hypersurface of degree $d\geq n$ and let $\BF(X)\subset \BG(\pp 1,\pp n)$ denote
the Fano scheme of lines on $X$. Let $  B\subset \BF(X)$ be
an irreducible component of maximal dimension. 
Let
$\ci_B :=\{ (x,E)\mid x\in X, E\in B, x\in \BP E\}$
and let $\pi,\rho$ denote the projections to
$X$ and $B$. Let $X_B=\pi(\ci_B)\subseteq X$ and let $\tilde \cC_x=\pi\rho\inv\rho\pi\inv(x)$.

  If $\tdim \BF(X)\geq n-2$, then
for all $x\in  X_B$,
$\tilde \cC_x\cap \sing X \neq\emptyset$.
\end{conjecture}

By taking hyperplane sections in the case $d=n$, Conjecture \ref{mainthm}  would imply the following conjecture,
which was conjectured independently by  O.  Debarre and J. de Jong:

 \begin{conjecture}[Debarre-de Jong conjecture]\label{d-dj} Let $X^{n-1}\subset \pp n$ be a smooth hypersurface
of degree $d\leq n$, then the dimension of the Fano scheme  of lines
on $X$  equals $2n-d-3$.
\end{conjecture}

Our conjecture extends to smaller degrees as follows:

\begin{conjecture}\label{thmb} Let $X^{n-1}\subset  \pp n$ be a
hypersurface of degree $n-\l$. Let $  B\subset \BF(X)$ be
an irreducible component of   dimension $n-2$, with $\ci_B,X_B$ etc\dots as above.
 If $\tcodim(X_B,X)\geq \l$ and $\cC_x$ is reduced for general $x\in X_B$, then
for all $x\in  X_B$,
$\tilde \cC_x\cap \sing X \neq\emptyset$.
\end{conjecture}

 The cases $X_B=X$ and $\tcodim(X_B,X)=\frac n2$ are known, e.g., they
appear in Debarre's unpublished notes containing Conjecture~\ref{d-dj}. 
In \cite{harrisetal},  J. Harris et. al.  proved Conjecture~\ref{d-dj} when $d$ is small with respect to $n$.
Debarre  in his unpublished notes also proved the case $d=n\leq 5$
 and
A. Collino \cite{collino} had earlier proven the case $d=n=4$. In \cite{beheshti},
R. Beheshti proved the case $d=n\leq 6$ and a different proof   was also given in \cite{LR}.  
 
 Conjecture \ref{mainthm} would also imply a special case of a conjecture of Beheshti and J. Starr
(Question 1.3 of \cite{Beh2}), about $\pp k$'s on hypersurfaces, which, in
the same paper, Beheshti proved for $k\geq (n-1)/4$ and Conjecture \ref{mainthm}   would prove 
for $k=1$.

Central to our work is finding additional structure on the tangent
space to $B$ at a general point. This structure gives rise to vector bundles
on $\cC_x$.
We   prove Conjecture \ref{mainthm} when 
the construction gives rise to 
exactly one vector bundle, see Theorem~\ref{m1lemma}.

\subsection{Overview} The statement of Conjecture~\ref{mainthm} indicates how one should look  for
singular points. Say $y\in X$ and we want to determine if $y\in \sing X$, i.e.,
if $v_0,v_1\hd v_n$ is a basis of $W$ with $y=[v_0]$, and $P$ an equation for $X$,  we
would need that all partial derivatives of local coordinates in $y$ vanish. 
This is expressed by the $n$ equations
$dP_y(v_1)=\cdots = dP_y(v_n)=0$.
If we fix a line $\BP E$ with $y\in \BP E\subset X$ and look
for a singular point of $X$ on $\BP E$, if $e_1,e_2$ is a basis of $E$ that
we expand to a basis $e_1,e_2,w_1\hd w_{n-1}$ of $W$, the equations
$dP_y(  e_1)=dP_y(e_2)=0$
 come for free, so we have one less equation to satisfy.

A further simplification is obtained by a study of $T_EB\subset T_EG(2,W)=E^*\ot W/E$.
We observe $T_EB$ is the kernel of the map 
$\a\ot w\mapsto \a\circ (w\intprod P)|_E$,
described in   Proposition \ref{tanspaceprop}. Moreover, we identify the tangent space $T_E\cC_x\subset T_EB\subset E^*\ot W/E$
to the Fano scheme of $B$-lines through $x$
as a subspace $\hat x^{\perp E}\ot \Pi$, where $\Pi\subset W/E$ is independent
of $x\in \BP E$, see Proposition \ref{tangcx}.
In the same Proposition we remark that $E^*\ot \Pi\subset T_EB$
is the intersection of $T_EB$ with the locus of rank $1$ homomorphisms in $T_EG(2,W)=E^*\ot W/E$.
As a consequence,  $T_EB/(E^*\ot \Pi)$ corresponds to a linear subspace of
the space of $2\times m$ matrices of constant rank two, for which there are normal forms.
The normal forms allow  us to reduce the number of equations defining the singular locus on a given line even
further, see \S\ref{sect3}. Now the number of equations we reduce by will depend on the  dimension of $\Pi$, but  it is always bounded by  $\tdim \tilde \cC_x$,
where $\tilde \cC_x$ is the cone swept by the lines of $B$ passing through a general point $x$. For this reason, one expects to find at least a finite number of singular points of $X$ lying on $\tilde \cC_x$.

In the second part of the paper we show that  $\sing X\cap\tilde\cC_x$ as the zero locus of a section of a vector bundle, see equation \eqref{stareqn}.
We determine certain positivity properties of the vector bundles we work with in Lemma \ref{genamplelemma}, observe an elementary case where $X$ must be singular (Theorem \ref{everyp1}) and prove
Theorem \ref{m1lemma}.

\subsection*{Acknowledgement} We thank IHES for providing us an outstanding environment for
collaborating on this project and Johan de Jong for useful comments on earlier versions of this paper. 
The second author would like to thank Remke Kloosterman for discussions during the preparation of this paper.

\section{The tangent space to $B$}\label{T_EB}
In this section we study $\pp k$'s on an arbitrary projective variety $X\subset \BP W$.
Let $W$ denote  a vector space  over an algebraically closed field $K$ of characteristic zero.
For algebraic subsets $Z\subset \BP W$, we let $\hat Z\subset W$ denote
the affine cone and  
$\BF_k (Z)\subset G(k+1 ,W)=\BG (\pp k, \BP W)$ denotes the Fano scheme of
$\pp {k}$'s on $Z$.  
Let $X \subset \BP W$ be a variety. Let $  B\subset \BF_k(X)$ be
an irreducible component.  Let
$\ci_B :=\{ (x,E)\mid E\in B, x\in \BP E\}$
be the incidence correspondence and let $\pi,\rho$ denote the projections to
$X$ and $B$. Let $X_B=\pi(\ci_B)$.
Let $\cC_x=\rho\pi\inv(x)$ and let 
$\tilde \cC_x=\pi\rho\inv(\cC_x)$, so
$\tilde \cC_x\subset X\subset \BP W$ is a cone with vertex $x$
and base isomorphic to $\cC_x$.

For a vector space $V$, $v\in V$, and $q\in S^kV^*$, we let
$v\intprod q\in S^{k-1}V^*$ denote the contraction. We also
write $q(v^a,w^{k-a})=q(v\hd v, w\hd w)$ etc. when
we consider $q$ as a multi-linear form. We denote the symmetric
product by $\circ$, e.g., for $v,w\in V$, $v\circ w\in S^2V$.
The following proposition is essentially a rephrasing of the discussion on
p. 273 of \cite{EH}. 
We include a short proof for the sake of completeness.

\begin{proposition}\label{tanspaceprop} Let $E\in \BF_k(X)$,
then $T_E\BF_k(X)=\tker\sxe$ where
\begin{equation}
\begin{aligned}
\sxe: T_EG(k+1,W)=E^*\ot W/E &\ra \oplus_d Hom(I_d(X), S^dE^*)\\
\a\ot w&\mapsto
 \{ P\mapsto  \a\circ (w\intprod P)|_E\}.
\end{aligned}
\end{equation}
\end{proposition}
 
\begin{proof}
We first note that $(w\intprod P)|_E$ is well defined because $P|_E=0$. Without loss of generality, we can restrict to the case where $X$ is a hypersurface defined by a degree $d$ polynomial $P$. The general case will follow by considering intersections.

Let $e_0,\dots,e_k$ be a basis of $E$, $\alpha_0,\dots,\alpha_k$ be the dual basis. A tangent vector $\eta=\alpha_0\otimes \bar w_0 + \cdots + \alpha_k\otimes \bar w_k\in T_E G(k+1,W)$ corresponds to the first order deformation $E_t = \langle e_0+tw_0,\dots,e_k+tw_k\rangle$ of $E$ in $W$, where the $w_j$ are arbitrary liftings of the $\bar w_j$ to $W$. Recall that $E=\langle e_0,\dots,e_k\rangle$ belongs to $\BF_k(X)$ if and only if $P$ vanishes on all points of $\BP E$, i.e., if and only if $P(e_0^{b_0},\dots,e_k^{b_k})=0$ for all $b_0,\dots,b_k$ such that $b_0+\cdots+b_k=d$. Therefore, the condition $\eta\in T_E\BF_k(X)$ is equivalent to the vanishing of   
$$
\begin{aligned}
&P((e_0+tw_0)^{b_0},(e_1+tw_1)^{b_1}\hd (e_k+tw_k)^{b_k})\\
&=
P(e_0^{b_0},\dots,e_k^{b_k})
+t[P(e_0^{b_0-1},w_0,e_1^{b_1}\hd e_k^{b_k})+\cdots + P(e_0^{b_0 }, e_1^{b_1}\hd e_k^{b_k-1},w_k)]\\
&=
0 + t [(\alpha_0\circ P)(w_0,e_0^{b_0},\dots,e_k^{b_k}) + (\alpha_1\circ P)(w_1,e_0^{b_0},\dots,e_k^{b_k})
+\cdots + (\alpha_k\circ P)(w_k,e_0^{b_0},\dots,e_k^{b_k})] \\
&=
t[(\sxe(\eta)) (e_0^{b_0},\dots,e_k^{b_k})] 
\end{aligned}
$$
in $K[t]/(t^2)$ for every choice of $b_0,\dots,b_k$. This implies the claim.
\end{proof}

\begin{proposition}\label{tangcx} Notations as above. Let  $x$ be a general point of
$X_B$ and $E$   a general point of $B$ with $x\in \BP E$.
\begin{enumerate}
\item If there exist 
     $w \in W/E$
and $\a\in E^*\backslash 0$ such that $\sxe (\a\ot w)=0$,
then $E^*\ot w\subset \tker\sxe$.

\item Letting $\Pi\subset W/E$ be maximal such that $E^*\ot \Pi\subset \tker\sxe$,
then for $x\in \BP E$, $T_E\cC_x=\hat x^{\perp E}\ot \Pi$.
\end{enumerate}
\end{proposition}

\begin{proof} $\sxe (\a\ot w)=0$ says 
$\a(u)P(w,u^{d-1})=0$ for all $u\in E$ and for all $P\in I(X)$.
If $P(w,u^{d-1})=0$ for all $u$ with $\a(u)\neq 0$, then $P(w,u^{d-1})=0$
for all $u\in E$, thus $E^*\ot w\subset\tker\sxe$.
 The second assertion is clear.
\end{proof}

\section{How to find singular points on $X$}\label{sect3}
We now specialize to the case $k=1$ and $X$ is a hypersurface  in $\pp n=\BP W$.
In this case $T_EB/(E^*\ot \Pi)$ is a linear subspace of
$K^2\ot K^m$ of constant rank two,
 where $m=n-1-\tdim T_E\cC_x$ in view of Proposition~\ref{tangcx}.
There is a normal form for linear subspaces $L$ of
$K^2\ot K^m$ containing no decomposable vectors.
Namely for every basis $\a^1,\a^2$   of $K^2$, 
there exist a basis
$w_1\hd w_m$   of $K^m$ and
  integers $s_1\hd s_r$, with 
$r = m-\tdim L$,
$s_1+\cdots +s_r= m$
and 
$s_1\geq s_2\geq \cdots \geq s_r\geq 1$
such that
\be
\begin{aligned}\label{norformeqn}
L=\langle &\a^1\ot w_1-\a^2\ot w_2,\a^1\ot w_2-\a^2\ot w_3
\hd \a^1\ot w_{s_1-1}-\a^2\ot w_{s_1},
\\
&
\a^1\ot w_{s_1+1}-\a^2\ot w_{s_1+2},
\a^1\ot w_{s_1+2}-\a^2\ot w_{s_1+3}\hd
\a^1\ot w_{s_2+s_1-1}-\a^2\ot w_{s_2+s_1},
\\
&
\dots
\\
&
\a^1\ot w_{s_{r-1}+\cdots +s_1+1}-\a^2\ot w_{s_{r-1}+\cdots +s_1+2}\hd
\a^1\ot w_{s_{r-1}+\cdots +s_1}-\a^2\ot w_{m}\rangle
\end{aligned}
\ene
The existence of this normal form is likely to be well known 
to the experts, although we could
not find an explicit reference. The proof is left to the reader.
Note that the normal form gives a basis of $L$ divided into $r$ blocks of length $s_1-1,\dots,s_r-1$. In particular, if for some index $j$ we have $s_j=1$, then the corresponding block is empty.

Applying this normal form, we obtain a normal form for $T_EB$.
Note that in this case 
$r=m-\tdim T_EB/(E^*\otimes\Pi)=n-1-\tdim T_EB+\tdim T_E\cC_x$. 
From now on, we will assume $\tdim B\geq n-2$, so $r\leq \tdim\cC_x+1$, 
with equality holding generically if $\tdim B=n-2$ and $B$ is reduced.

\begin{lemma}\label{normalform} Let $X\subset \BP W$ be as above and assume
$\tdeg(X)=d \geq 1+s_1$.
 Let $E$ be a general point of $B$. 
Then there exist
$p^E_j\in S^{d- s_j}E^*$, $1\leq j\leq r$ such that
$$
\tim\sxe=
S^{s_1}E^*\circ p^E_1+\cdots
+S^{s_r}E^*\circ p^E_r 
$$
where $w_1\hd w_{n-1}$ is a basis of $W/E$ such that
$\Pi=\langle w_{m+1}\hd w_{n-1}\rangle$ and $w_1\hd w_m$ are adapted to
the normal form \eqref{norformeqn}.

\end{lemma}

We remark that here and in Lemma \ref{lemmabelow} below, one can drop the
assumption that $E$ is a general point of $B$. The only change at special
points is that the normal form \eqref{norformeqn} will be different.
\begin{proof}
Apply the normal form to $\tker\sxe/(E^*\ot\Pi)$.
For $1\leq j\leq s_1-1$ we have
\be\label{relationeqna}
\a^1\circ (w_j\intprod P)|_E = \a^2\circ (w_{j+1}\intprod P)|_E
\ene
Since $\a^1,\a^2$ are linearly independent, for $j=1$ this
implies there exists $\phi_1\in S^{d-2}E^*$ such that
$(w_1\intprod P)|_E=\a^2\circ \phi_1$ and 
$(w_2\intprod P)|_E=\a^1\circ \phi_1$.
But for the same reason, when $j=2$ we see there exists
$\phi_2\in S^{d-3}E^*$ such that
\be\label{relationeqnb}
(w_1\intprod P)|_E=(\a^2)^2\circ \phi_2,
\ \ \ 
(w_2\intprod P)|_E=(\a^1\circ \a^2) \circ \phi_2,
\ \  \ 
(w_3\intprod P)|_E=(\a^1)^2\circ \phi_2
\ene
and so on until we arrive at $\phi_{s_1-1}=:p_1^E\in S^{d-s_1}E^*$, such
that $(w_j\intprod P)|_E=(\a^1)^{j-1}(\a^2)^{s_1-j}p_1^E$
for $1\leq j\leq s_1$.
In particular,
$S^{s_1}E^*\circ p_1^E\subset\tim\sxe$.
Continuing in this way for the other chains in the normal form
we obtain polynomials $p_1^E\hd p_r^E$ 
with $S^{s_k}E^*\circ p_k^E\subset \tim \sxe$.
Note that if $s_k=1$ we set $p_k^E=(w_{s_{k-1}+\dots+s_{1}+1}\intprod P)|_E$.
\end{proof}

Note that without assumptions on the degree,
the conclusion of Lemma \ref{normalform} can fail, e.g., 
if $d=3$ and $s_1=m=3$, as in the case of a general cubic hypersurface, 
then \eqref{relationeqnb} only says $(w_1\intprod P)|_E= (\a^2)^2$,
$(w_2\intprod P)|_E= \a^1\circ \a^2  $ and 
$(w_3\intprod P)|_E=(\a^1)^2$. This does imply that the   image of $\BP E $
under the Gauss map of $X$
is a rational normal 
curve of 
degree two in $\BP (E+\Pi)\upperp\subset \BP W^*$ and one can obtain
similar precise information about the Gauss image of $\BP E$ in other cases.
However, as long as $\tdeg(X)\geq s_1+1$ 
the conclusion holds.

When $s_1=n-1-\tdim \cC_x$ there is  a single polynomial  on $\BP E$
whose zero set corresponds to singular points of $X$. Thus:
 
\begin{theorem}\label{everyp1} 
Let $X^{n-1}\subset \pp n$ be a hypersurface, with $B,\tilde\cC_x$ etc. as above.
If   $\tdeg(X)\geq s_1+1$ and $s_1=n-1-\tdim\cC_x$, then
for all $E\in B$, $\BP E\cap \sing X\neq\emptyset$.
\end{theorem}

\begin{lemma}\label{lemmabelow} Let $X$ be as above and let $E$ be a general point of $B$.
Write $\{\tdeg\ p_k^E: 1\leq k\leq r\}=\{\d_1<\d_2<\cdots<\d_c\}$ and set $i_j=\#\{p_k^E: \tdeg\ p_k^E\leq \d_j\}$ for all $j\leq c$. Note that if $s_r=1$ in the normal form \eqref{norformeqn}, then $i_{c-1}=\#\{k: s_k> 1\}$ and $i_c=r$.
Consider the vector spaces
\begin{align*}
\hat M_1&=M_1:=\langle p^E_1\hd p^E_{i_1}\rangle\subset S^{\d_1}E^*\\
\hat M_2&:=\langle p^E_{i_1+1}\hd p^E_{i_2},\hat M_1\circ S^{\d_2-\d_1}E^*\rangle\subset S^{\d_2}E^*\\
M_2&:= \hat M_2/(\hat M_1\circ S^{\d_2-\d_1}E^*) \subset S^{\d_2}E^*/(\hat M_1\circ S^{\d_2-\d_1}E^*)\\
&\vdots \\
\hat M_{c-1}&:=\langle p^E_{i_{c-2}+1}\hd p^E_{i_{c-1}}, \hat M_{c-2}\circ   S^{\d_{c-1}-\d_{c-2}}E^*\rangle\subset S^{\d_{c-1}}E^*\\
M_{c-1}&:= \hat M_{c-1}/(\hat M_{c-2}\circ   S^{\d_{c-1}-\d_{c-2}}E^*)   \subset S^{\d_{c-1}}E^*/(\hat M_{c-2}\circ   S^{\d_{c-1}-\d_{c-2}}E^*)\\
\hat M_c&:=\langle p^E_{i_{c-1}+1}\hd p^E_{i_c}, \hat M_{c-1}\circ S^{\d_c-\d_{c-1}}E^* \rangle
\subset S^{\d_{c}}E^*\\
M_{c}&:= \hat M_c/(\hat M_{c-1}\circ S^{\d_c-\d_{c-1}}E^*)
\subset S^{\d_{c}}E^*/(\hat M_{c-1}\circ S^{\d_c-\d_{c-1}}E^*)
\end{align*}
These spaces are well defined and depend only on $X,E$.
\end{lemma}

The lemma is an immediate consequence of the uniqueness of the normal form up to admissible
changes of bases.
Let $I_E\subset Sym(E^*)$ denote the ideal generated by the $\hat M_j$.
Note that the number of polynomials generating $I_E$ is at most   $ \tdim \cC_x+1$, independent of the
normal form (and $ \tdim \cC_x+1$ is the expected number of generators).
Let $B'\subset B$ denote the Zariski open subset where the normal form
is the same as that of a general point.

\begin{proposition}\label{psingprop}
Let $E\in B'$ and let  $[y]\in \BP E$ be in the zero set of $I_E$, then $[y]\in \sing X$.
\end{proposition}

\begin{proof}
$[y]\in \sing X$ says that for all $w\in W$, $(w\intprod P)(y)=0$. 
Let $w_1,\dots,w_{n-1}$ be elements of $W$ that descend to give a basis of $W/E$. 
Since $(u\intprod P)|_E=0$ holds for all $u\in E$, the polynomial $(w\intprod P)|_E\in S^{d-1}E^*$ is a linear combination of the $(w_i\intprod P)|_E$. As each $(w_i\intprod P)|_E$ contains one of the $p^E_j$ as a factor, the hypothesis implies $w\intprod P$ vanishes at $y$.
\end{proof}

\medskip

We now     allow $E$ to vary. 
Let $\cS\ra G(2,W)$ denote  the tautological rank two subspace
bundle and note  that the total space of $\BP \cS|_B$ is our incidence correspondence $\ci_B$.
Since all calculations are algebraic,
$M_1$ gives rise to a rank $i_1$ algebraic vector bundle $\cM_1\subset S^{\d_1}\cS^*|_{B'}$,
and $M_2$ gives rise to a rank $i_2-i_1$ algebraic vector bundle $\cM_2\subset ((S^{\d_2}\cS^*)/(\cM_1\circ S^{\d_2-\d_1}\cS^*))|_{B'}$, etc\dots  finally giving a bundle of ideals $\ci\subset Sym(\cS^*)|_{B'}$.

Now, since Grassmannians are compact, along any curve $E_t$ in $B$
with $E_t\in B'$ for $t\neq 0$, we have well defined limits as $t\ra 0$, and
thus we may define $\bold I^{E_t}_0 \subset Sym(E_0^*)$. Note that if
we approach $E_0$ in different ways, we could obtain different limiting ideals,
nevertheless we  have:

\begin{proposition}\label{psingpropb}
Let $E\in B$ and let  $[y]\in \BP E$ be in the zero set of $\bold I^{E_t}_0$, then $[y]\in \sing X$.
\end{proposition}

\begin{proof} Although this is a standard argument, we give   details
in a special case to show that
at   points of $B\backslash B'$ the situation is even more favorable. 
Work locally in a coordinate patch. First note that we may choose
a fixed $\a^1,\a^2\in W^*$ that restrict to a basis of $E^*$ for
all $E$ in our coordinate patch and still obtain the normal form
by linear changes of bases in $W/E$. So along our curve $E_t$ we
consider $\a^1,\a^2$ and $w^t_1\hd w^t_{n-1}$ such that for $t\neq 0$
(and small), $\Pi=\langle  w^t_{m+1}\hd w^t_{n-1}\rangle $ and we have a fixed normal form,
e.g.,  say
$\a^1\ot w^t_1-\a^2\ot w^t_2,\a^1\ot w^t_3-\a^2\ot w^t_4\in \tker\s_t$ for
all small $t$, giving rise to polynomials $\phi_t,\psi_t$ such that
$$
w^t_1\intprod P|_{E_t}=\a^2\circ \phi_t,\ \ 
w^t_2\intprod P|_{E_t}=\a^1\circ \phi_t,\ \ 
w^t_3\intprod P|_{E_t}=\a^2\circ \psi_t,\ \ 
w^t_4\intprod P|_{E_t}=\a^1\circ \psi_t.
$$  
In the limit, we may not assume that
$w^0_1\hd w^0_m$ are linearly independent.

First notice that if $\psi_0=\mu \phi_0$, then
although we have a well defined plane $\tlim_{t\ra 0}[\phi_t\ww \psi_t]$ (which equals
$[\phi_0\ww(\psi_0'-\mu\phi_0')]$ if $\phi_0\ww(\psi_0'-\mu\phi_0')\neq 0$), the vanishing of $\phi_0$
already implies $[y]\in \sing X$, as long as $w^0_1\hd w^0_4$ 
are linearly independent.

Now consider the case we have a relation $\l^1w^0_1+\cdots + \l^4w^0_4=0$. 
This implies we have a relation
\begin{align*}
0&=\l_1\a^2\circ \phi_0+\l^2\a^1\circ \phi_0+\l^3\a^2\circ \psi_0+\l^4\a^1\circ \psi_0\\
&=\a^1\circ (\l^2\phi_0+\l^4\psi_0)+\a^2\circ (\l^1\phi_0+\l^4\psi_0)
\end{align*}
Which implies (assuming all coefficients nonzero)
$\psi_0=\mu\phi_0$ with $\mu=-\l^2/\l^4=-\l^1/\l^3$.
In particular, the relation  among the $w^0_j$  was not arbitrary.
We also see that
$$
(\l^1{w^0_1}'+\cdots +\l^4{w^0_4}')\intprod P|_{E_0}=(\a^1+\mu\a^2)(\l^2{\phi_0}'+\l^4{\psi_0}')
$$
That is, assuming $z:=(\l^1{w^0_1}'+\cdots +\l^4{w^0_4}')$ is linearly independent
of $w^0_1\hd w^0_4$, we obtain that    $\bold I^{E_t}_0$
includes $z\intprod P|_{E_0}$. Otherwise, just differentiate further.
\end{proof}

We would like to work with vector bundles over our entire space, which can be achieved by
considering the product of Grassmann bundles $ G(\trank \hat M_1,S^{\d_1}\cS^*)
\ctimes G(\trank \hat M_c,S^{\d_c}\cS^*)\ra B$.
Over $B'\subset B$ we have a well defined section of this bundle. Using the compactness
of the Grassmannian and the limiting procedure described above, we
extend these sections to obtain a space $\tau :\cB\ra B$, with fiber over points of $B'$
a single point. 
Thus each $M_j$ (resp. $\hat M_j$) gives rise to a well defined vector bundle $\bold M_j\ra \cB$
(resp. $\hat{\bold M}_j\ra \cB$), and
we have the corresponding bundle of ideals $\bold I\subset \tau^*(Sym(\cS^*))$.

Let $\bold S=\t^*(\cS)$ and
  $\cO_{\BP(\bold S)}(\d )=\tilde \t^*(\cO_{\BP \cS}(\d))$, 
where
$\tilde\t: \bold S\ra \cS$ is the lift of $\t$.   Consider
the projection  $q: \BP(\bold S)\ra \cB$ and
the bundles 
$$   q^*({\bold M_j})^*\ot \cO_{\BP(\bold S)}(\d_j) 
$$
Then $q^*({\bold M_1})^*\ot \cO_{\BP(\bold S)}(\d_1)= q^*(\hat{\bold M}_1)^*\ot \cO_{\BP(\bold S)}(\d_1)$ has a canonical section $\bold s_1$ whose zero set
$Z_1\subset \BP(\bold S)$
is the zero set of $(\bold I)_{\d_1}$.
For each $2\leq j\leq c$,  
   the corresponding bundle $\htmi$,  
  has a canonical section $ \hat {\bold s}_j$, whose zero set
$Z_j\subset \BP(\bold S)$ is the zero set of $(\bold I)_{\d_j}$.

Fix a general point $x\in X_B$, let $\bcx=\t\inv(\cC_x) \subset \cB$.
The essential observation is that 
$\tdim\tilde\cC_x\geq r = \sum_j\trank \bold M_j$,
so we expect $Z_{c}\cap q\inv ({\bold C_x})$ to be nonempty.
This would imply the existence of singular points, because the image 
of $Z_{c}$ in $X_B$ is contained in $\sing X$.

In more detail, 
we have a sequence of vector bundles $\tmii 1\hd \tmii c$ over $\BP (\bold S)$,
whose ranks add up to $r$, such that $\tmii 1$ is equipped
with a canonical section $\bold s_1$, and restricted to its zero set $Z_1$,
$\tmii 2$ has a canonical section $\bold s_2$, etc... such that if everything were to work
out as expected, the zero set $Z_c$ of $\bold s_c$, which is defined as a section of
$\tmii c$ over $Z_{c-1}$, would have codimension $r$, which is the dimension of
$\BP (\bold S)|_{\bold C_x}$.
Thus we expect $Z_c\cap  \BP (\bold S)|_{\bold C_x}\neq\emptyset$,
which would imply that $\tilde\cC_x\cap \sing X\neq\emptyset$.
Note that a sufficient condition for this is
\begin{equation}\label{stareqn}
c_{\ttop}(\tmii 1)\cdot c_{\ttop}(\tmii 2)\cdots c_{\ttop}(\tmii c)\neq 0,
\end{equation}
where the intersection takes place in the Chow group of codimension $r$ cycles on $\BP(\bold S)|_{\bold C_x}$.
 
We were not able to prove this in general, but we are able to show:

\begin{theorem}\label{m1lemma}
The zero set of the canonical section of 
$\hat {\bold M}_1^*\ot \t^*(\cO_{\bpwx}(\d_1))|_{\bold C_x}$ is always at least of the expected dimension.
\end{theorem}

Another natural case to consider is the case where the $\bold M_j$ are all line bundles. 
For instance, consider the 
even further special case where there is just $\bold M_1,\bold M_2$ and both are line bundles.
This case splits into two sub-cases, based on whether or not the zero section of $\bold s_1$ surjects onto
all of $X_B$ or not. 
In \S 6, we show that if $Z(\bold s_1)$ fails to surject onto  $X_B$, 
then Conjecture \ref{mainthm}  indeed holds.

Since   $\tmi$ only has a section defined over $Z_{j-1}$, it will
be more convenient to work with the bundles 
$q^*(\hat{\bold M}_j)^*\ot \cO_{\BP(\bold S)}(\d_j)$ which have everywhere defined sections $\hat{\bold s}_j$.

The best situation for proving results about sections of bundles is
when the bundles are ample, which fails here. However,
below we show that if $x$ is sufficiently general, the bundles  $\hat {\bold M}_j^*\ot \t^*(\cO_{\bpwx}(\d_j))|_{\bold C_x}$
are \emph{generically ample}.

\section{Generic ampleness}
 
Recall (\cite{Fulton}, Example 12.1.10) that a vector bundle $\cE$ over a variety $X$ is {\it generically
ample} if it is generated by global sections and the
canonical map $\BP \cE^*\ra \BP (H^0(X,\cE)^*)$ is generically finite.
The locus where it is not finite is called the {\it disamplitude locus} $\Damp(\cE)$.
In particular, if $Y\subset X$ is a subvariety such that
$\cE|_Y$ has a trivial quotient sub-bundle, then $Y\subset \Damp(\cE)$.

Generically ample bundles of rank $r\leq\tdim X$ 
have the property
that $c_1(\cE)\hd c_r(\cE)$ are all positive, in the sense that
their classes in the Chow group of $X$ are linear combinations of effective classes with nonnegative coefficients, not all equal to $0$.

\begin{lemma}\label{genamplelemma} For
general $x\in X_{B}$, the bundles 
 $\hat {\bold M}_j^*\ot \t^*(\cO_{\bpwx}(\d_j))|_{\bold C_x}$ are generically ample.
\end{lemma}

\begin{proof} 
First, global generation is clear, as for all the $\hat {\bold M}_j$
we have a surjective map
$$
 S^{\d_j}\bold S\ot \t^*(\cO_{\bpwx}(\d_j))|_{\bold C_x}\ra \hat {\bold M}_j^*\ot  \t^*(\cO_{\bpwx}(\d_j))|_{\bold C_x}.
$$

Now take any choice of splitting $W=\hat x\op W'$, so the left
hand side becomes 
$\t^*(\cO_{\BP(W/\hat x)}\oplus\cO_{\BP(W/\hat x)}(1)\oplus\cdots\oplus \cO_{\BP(W/\hat x)}(\d_j))$
restricted to $\bold C_x$.
Each factor is globally generated and this of course
remains true when restricting to subvarieties.

\smallskip

The locus where the canonical map
$$\BP(\oplus_{i=0}^{\d_j}\cO_{\bpwx}(-i))
\ra \BP (H^0(\bpwx, \oplus_{i=0}^{\d_j}\cO_{\bpwx}(i))^*)
$$
is not finite is the $\BP \cO_{\bpwx}$ factor.
Hence, when we restrict to $\cC_x$ and pull-back to $\bold C_x$, $\Damp(\hat {\bold M}_j^*\ot \t^*(\cO_{\bpwx}(\d_j))|_{\bold C_x})$ is contained in the union of the following two loci:\begin{itemize}
\item the locus where the map $\t\co \bold C_x\rightarrow \cC_x$ has positive-dimensional fibers;
\item the projection to 
$\bold C_x$ of the locus where the image of 
$$\BP(\hat {\bold M}_j\ot  \t^*(\cO_{\bpwx}(-\d_j))|_{\bold C_x})
\rightarrow 
\BP(\oplus_{i=0}^{\d_j}\t^*\cO_{\bpwx}(-i))
$$
 intersects $\BP(\t^*(\cO_{\bpwx}))$.
\end{itemize}  
The lemma will  follow  from
Lemma~\ref{notallP} below and the fact that 
   the general fiber of $\bold C_x\rightarrow \cC_x$ is finite if $x$ is general. 
Note that the image of $\BP\left(\hat{\bold M}_j\ot  \t^*(\cO_{\bpwx}(-\d_j))|_{\bold C_x}\right)$ inside 
$\BP(\oplus_{i=0}^{\d_j}\t^*\cO_{\bpwx}(-i))$
intersects $\BP(\t^*\cO_{\bpwx}|_{\bold C_x})$ precisely over the points $E\in\bold C_x$ such that the fiber $\hat{\bold M}_{j,E}$ contains a nonzero polynomial vanishing at $x$ with multiplicity $\d_j$. 
\end{proof}

\begin{lemma}\label{notallP} For general $x\in X_{B}$
 and general $E\in \cC_x$,
all nonzero elements $P\in  (\bold I_E)_k$ vanish
at $x$ with multiplicity at most $k-1$
for any integer $k\leq \delta_c$.
\end{lemma} 
\begin{proof}  
Fix $E\in B$. Then
the locus
$$\{[P]\in \BP((\bold I_E)_k)\, \mid \,  P=f^k\text{ for some }f\in E^*\}$$
is the intersection of $\BP((\bold I_E)_k)$ with a degree $k$ rational normal curve contained in $\BP(S^kE^*)$. Hence, it consists of at most a finite number of points $[P_1],\dots,[P_R]$. Thus it suffices to choose a point $x\in \BP E$ such that $P_j(x)\neq 0$ for all $j=1,\dots,R$.
\end{proof} 

\section{Proof of Theorem~\ref{m1lemma}}

Lemma~\ref{m1lemma} is a consequence of  Lemma~\ref{genamplelemma} for $j=1$, combined with the following lemma
with $M=\hat {\bold M}_1|_{\bold C_x}$.

\begin{lemma}\label{M1first}
Let $M\subset S^p\bold S^*|_{\bold C_x}$ be a vector bundle 
such that $M^*\otimes\cO_{\bold C_x}(p)$ is generically ample. Then the zero locus   of the canonical section of $q^*M^*\otimes\cO_{\BP\bold S|_{\bold C_x}}(p)$ is of dimension
at least $\tdim \bold C_x+1-\trank (M)$.
\end{lemma}

The proof of Lemma \ref{M1first} follows by several reductions
which reduce the question to a basic fact about
intersections on nontrivial $\pp 1$-bundles over a curve:

\begin{lemma}\label{itcone} 
Let $p_\xi\co S\rightarrow \xi$ be a $\BP^1$-bundle over a curve $\xi$,
with a section $e\co \xi \rightarrow S$ of negative self-intersection. 
If $\tilde D_1,\tilde D_2$ are effective divisors of $S$ not contained in the image of $e$   such that the restriction of $p_\xi$ to each of them is finite,
then $\tilde D_1\cap \tilde D_2\neq\emptyset$.
\end{lemma}

\begin{proof} 
The Picard group of $S$ is generated by the class $\xi_0$ of the image of $e$ and the class $F$ of a fiber of $p_\xi$. Since $S$ is not a product, one has $F^2=0$, $\xi_0\cdot F=1$ and $\xi_0^2=-k$ with $k$ a positive integer. Choose irreducible components $D_1$, $D_2$ of the divisors, different from the image of $e$. Then $D_i=a_i\xi_0+b_iF$ with $a_i\geq 1$ (since it is the degree of $p_\xi|_{D_i}$) and $D_i\cdot \xi_0 = b_i-a_ik\geq 0$. Then $D_1\cdot D_2=-a_1a_2k+a_1b_2+a_2b_1\geq a_1a_2k>0$. From this the claim follows.
\end{proof}

The proof of Lemma \ref{M1first} relies on the following
Lemma:

\begin{lemma}\label{hitssubbundleslemma}
Let $M\subset S^p\bold S^*|_{\bold C_x}$ be a vector bundle 
such that $M^*\otimes\cO_{\bold C_x}(p)$ is generically ample. 
Let $W'\subset W$ be any hyperplane not containing $\hat x$,   set $H'=\cC_x\cap \BP W'$, and let $H\subset\BP(\bold S|_{\bold C_x})$ be the preimage of $H'$ under $\tau\co\BP(\bold S|_{\bold C_x})\rightarrow\tilde\cC_x$.
Let $\bold s_M$ denote the canonical section of $q^*M^*\otimes\cO_{\BP\bold S|_{\bold C_x}}(p)$.

Then the intersection $Z(\bold s_M)\cap H$ has  dimension
at least $\tdim \bold C_x-\trank (M)$. 
In particular, it is nonempty if $\trank M \leq \tdim \bold C_x$. 
\end{lemma}

\begin{proof}
Consider the section $s_{M,W'}\in H^0(H',q^*\otimes\cO_{\BP(\bold S|_{\bold C_x})}(p))$ obtained by restricting $\bold s_M$ to $H'$. Then we have 
$Z(\bold s_M)\cap H= Z(s_{M,W'})$.

Observe that $\rho\co\tilde\cC_x\rightarrow \cC_x$ and $q\co\BP(\bold S|_{\bold C_x})\rightarrow \bold C_x$ become isomorphisms when restricted to, respectively, $H'$ and $H$. In particular, since $H'$ was an hyperplane section of $\cC_x$, the isomorphism $H\cong\bold C_x$ so obtained induces an isomorphism $\cO_{\BP(\bold S|_{\bold C_x})}(1)|_H\cong\cO_{\bold C_x}(1)$. Since the isomorphism $H\cong\bold C_x$ also induces an isomorphism $(q^*M)|_H\cong M$, one can view $s_{M,W'}$ as a global section of $M^*\otimes\cO_{\bold C_x}(p)$. Therefore, if $Z(s_{M,W'})\subset H$ is nonempty, it has codimension at most $\trank M$ in $H$ \cite[Prop. 14.1b]{Fulton}. It remains to show $Z(s_{M,W'})\neq\emptyset$ if $\trank M\leq\tdim\bold C_x$. 

Recall from \cite[\S 14.1]{Fulton} that there is a localized Chern class associated to the section $s_{M,W'}$, which is a class in the Chow group of $Z(s_{M,W'})$ whose pull-back under the inclusion $Z(s_{M,W'})\rightarrow \cC_x$ is the top Chern class of $M^*\otimes\cO_{\bold C_x}(p)$. 
Since $M^*\otimes\cO_{\bold C_x}(p)$ is generically ample and of rank $\leq\tdim\bold C_x$, its top Chern class is positive. So the Chow group of $Z(s_{M,W'})$ contains a nontrivial class, and in particular $Z(s_{M,W'})$ cannot be empty.
\end{proof}

\begin{proof}[Proof of Lemma~\ref{M1first}]
For every $E\in\bold C_x$, consider $N_E:=(S^{p-1}E^*\circ\hat x\upperp) \cap M_E$, the linear subspace of $M_E$ of forms vanishing on the point $x$. Without loss of generality, when $E$ varies $N_E$ gives rise to a vector subbundle $N\subset M$ of codimension 1. Indeed, if it were not so, there would be a point $E\in\bold C_x$ such that $N_E=M_E$, and then  $(E,x)$ would be a point of the zero locus
of the canonical section, thus implying the claim.

We have an exact sequence $0\rightarrow N \rightarrow M\rightarrow L \rightarrow 0$ where $L$ is the quotient line bundle. 
Since $q^*N^*\otimes\cO_{\bold C_x}(p)$ is a corank $1$ quotient of $q^*M^*\otimes\cO_{\bold C_x}(p)$, we can apply Lemma~\ref{hitssubbundleslemma} to it. 
Hence, the zero locus of the canonical section of $q^*N^*\otimes\cO_{\bold C_x}(p)$ contains an irreducible component $Z$ which intersects all subvarieties $H\subset \BP(\bold S)|_{\bold C_x}$ which come from preimages of general hyperplane sections of $\tilde\cC_x$.

Without loss of generality, we may assume that $Z$ is of dimension $1$, and that $q':=q|_Z:\;Z\rightarrow q(Z)=:\xi$ is a finite surjective map.
Recall that the group of Weil divisors (up to numerical equivalence)
of the ruled surface $\BP(\bold S)|_{\xi}$ is generated by the class $\xi_0$ of the tautological section of $q'$ (i.e., $(\xi_0)_E=(E,x)$) and the class $F$ of a fiber of $q$. From the effectivity of $Z$ and from Lemma~\ref{hitssubbundleslemma} we obtain $Z\cdot F\geq 1$, $Z\cdot \xi_0\geq 0$. 

To prove the claim, it suffices to show $Z\cdot c_1(q^*L^*|_{\xi}\otimes\cO_{\BP\bold S|_{\xi}}(p))>0$. We have $  c_1(q^*L^*|_{\xi}\otimes\cO_{\BP\bold S|_{\xi}}(p))\cdot F=d$
because
$
c_1(q^* L^*\otimes \cO_{\BP\bold S|_{\xi}}(p))\cdot F =
c_1(q^*L^*)\cdot F + c_1(\cO_{\BP\bold S|_{\xi}}(p))\cdot F = 0 + p = p
$.
Recall that the canonical section of $q^*N^*|_{\xi}\otimes\cO_{\BP\bold S|_{\xi}}(p)$ vanishes on $\xi_0$ by construction. Therefore, the canonical section of $q^*M^*|_{\xi}\otimes\cO_{\BP\bold S|_{\xi}}(p)$ induces a section $s_L$ of $q^*L^*|_{\xi}\otimes\cO_{\BP\bold S|_{\xi}}(p)$ on $\xi_0$. Since $N_E\subsetneq M_E$ for every $E\in\bold C_x$, we have that $s_L$ cannot vanish identically on $\xi_0$. Hence $c_1(q^*L^*|_{\xi}\otimes\cO_{\BP\bold S|_{\xi}}(p))\cdot \xi_0 \geq 0$, because it is the class of $Z(s_L)$ on $\xi_0$. Then the asserted inequality follows from Lemma~\ref{itcone} because $c_1(q^*L^*|_{\xi}\otimes\cO_{\BP\bold S|_{\xi}}(p))$ is linearly equivalent to an effective divisor satisfying the hypotheses of Lemma~\ref{itcone}.
\end{proof}

\section{Two line bundles}\label{twolb}

In this section, we prove the following result, which was announced in section~\ref{sect3}.

\begin{lemma}\label{twolblemma}
Assume $\bold M_1$ and $\bold M_2$ are line bundles, and that the projection $Z(\bold s_1)\rightarrow X_B$ is not surjective. Then for every $x\in X_B$, the zero set of ${\bold s}_2|_{\bold C_x}$ is of codimension at most $2$ in $\bold C_x$.
\end{lemma}

Therefore, we assume $\bold M_1$ and $\bold M_2$ are both line bundles. 

As in the arguments above, it will be sufficient to work with a general point $x\in X_B$ and a sufficiently general irreducible curve $\xi
 \subseteq \bold C_x$   and show that the
zero set of $\hat{\bold s}_2$ restricted to $\BP(\bold S)|_\xi$ is nonempty.
The proof is based on showing that $Z(\hat{\bold s}_2)\cap\BP(\bold S)|_\xi$ coincides with the zero set of the canonical section of $(q|_\xi)^*N^*\otimes\cO_{\BP(\bold S)|_\xi}(\d_2)$
a rank 2 vector bundle $N\subset S^{\d_2}\bold S|_\xi$ satisfying the hypotheses of Lemma~\ref{M1first}.
We construct $N$ under the assumption that the zero set $Z(\bold s_1)$ does not intersect the tautological section of $\BP(\bold S)|_\xi\rightarrow\xi$. 
 
Since $\bold M_1$ is a line bundle, we have that
$Z(\bold s_1)\subset\BP(\bold S)$ intersects every fiber of
$\BP(\bold S)\rightarrow\cB$ in $\d_1$ points, counted with multiplicity.
This follows from the very construction of the canonical section $\bold s_1$.

Without loss of generality in the choice of $x$ and $\xi$, we may assume:\begin{enumerate}
\item $\cO_{\bold C_x}(1)$ restricts to a generically ample line bundle $\cO_\xi(1)$ on $\xi$.
\item\label{d} $Z:=Z(\hat{\bold s}_1)\cap\BP(\bold S)|_\xi$ is not contained in the tautological section $\xi\rightarrow\BP(\bold S)|_\xi$. 
\item\label{t} $\hat{\bold M}_2^*\otimes\cO_{\bold C_x}(\d_2)$ is generically ample when restricted to $\xi$.
\item\label{q} the map $q|_Z\co Z\rightarrow \xi$ is finite. 
\end{enumerate}
The first assumption follows from the fact that $\bold C_x\rightarrow \cC_x$ is generically finite, so $\bold C_x\not\subset\Damp(\cO_{\bold C_x}(1))$ and the same holds for a generical $\xi\subset\bold C_x$.
Assumption~\eqref{d} follows from the genericity of $x$, and \eqref{t} follow from Lemma~\ref{genamplelemma}. 
Finally, if \eqref{q} did not hold, 
$Z(\hat{\bold s}_2)$ would  contain $\d_2$ points on every $1$-dimensional fiber of $q|_Z$ (counted with multiplicity), thus showing $Z(\hat{\bold s}_2)\neq\emptyset$.

For the rest of this section, we will often omit the restriction to $\xi$ from our notation.

Recall we have short exact sequence:
$$
0\rightarrow S^{\d_2-\d_1}\bold S^*\circ \bold M_1\rightarrow \hat {\bold M}_2\rightarrow \bold M_2\rightarrow0.
$$

As a consequence, the section $\hat{\bold s}_2\in H^0(\BP(\bold S),q^*\hat{\bold M}_2^*\otimes\cO_{\BP(\bold S)}(\d_2))$ canonically induces  a section $s\in H^0(Z,q^*{\bold M}_2^*\otimes\cO_{\BP(\bold S)}(\d_2))$. Assume that $Z(\hat{\bold s}_2)=\emptyset$, i.e., $Z(s)=\emptyset$ on $Z$. Then $s$ induces a trivialization $q^*{\bold M}_2^*|_Z\otimes\cO_{Z}(\d_2)\cong\cO_Z$. 

In this set-up, Lemma~\ref{twolb} is equivalent to the following lemma:
\begin{lemma}
Assume $Z$ does not intersect the image of the tautological section $s_0\co \xi\rightarrow \BP(\bold S)|_\xi$.
Then $Z(\hat{\bold s}_2)\cap\BP(\bold S)|_\xi\neq\emptyset$.
\end{lemma}

\begin{proof}
Assume by contradiction $Z(\hat{\bold s}_2)\cap\BP(\bold S)|_\xi$ is empty.
Fix a line $E\in\xi$. The fiber $\hat{\bold M}_{2,E}$ is spanned by all degree $\d_2$ multiples of polynomials in $\bold M_{1,E}$ and by an additional polynomial $\phi$, which does not vanish on any point of $Z$. 

Recall that no nonzero polynomial in $\bold {M}_{1,E}$ vanishes at $x$. 
Therefore, the condition of vanishing at $x\in\BP(E)$ with multiplicity $\d_2-\d_1$ defines a $1$-dimensional subspace of $S^{\d_2-\d_1}\bold S^*\circ \bold M_{1,E}$, and (for dimensional reasons) a $2$-dimensional subspace $N_E$ of $\hat{\bold M}_{2,E}$.
Hence, without loss of generality we may assume that $\phi$ is a polynomial vanishing at $x$ with multiplicity $\d_2-\d_1$. 
If we let $E$ vary, then $N_E$ defines a rank $2$ vector subbundle $N\subset \hat{\bold M}_2\subset S^{\d_2}\bold S^*$. Moreover, we have $N\otimes \cO_\xi(-\d_2+\d_1)\subset S^{\d_1}\bold S^*$. This follows from the fact that the condition of vanishing at $x$ with multiplicity at least $k$ defines the subbundle $\cO_\xi(k)\oplus \cdots \oplus \cO_\xi(\d_2)\subset \cO_\xi\oplus \cO_\xi(1)\oplus \cdots \oplus \cO_\xi(\d_2)\cong S^{\d_2}\bold S^*$. 
For every $E\in \xi$, if we choose any $0\neq\eta\in E^*$ that vanish on $x\in\BP(E)$, we have that  the fiber of $N\otimes \cO_\xi(-\d_2+\d_1)\subset S^{\d_1}\bold S^*$ over $E\in\xi$ is the locus of degree $\d_1$ polynomials $\psi$ over $\BP(E)$ satisfying $\eta^{\d_2-\d_1}\circ\psi\in \hat{\bold M}_{2,E}$.

By the description of the fibers of $\hat{\bold M}_2$ given above, all points in the zero locus of the canonical section of $q^*(N\otimes \cO_\xi(-\d_2+\d_1))\otimes \cO_{\BP(\bold S)|_\xi}(\d_1)$ belong to $Z(\hat{\bold s}_2)$. Hence, the canonical section of $q^*(N\otimes \cO_\xi(-\d_2+\d_1))\otimes \cO_{\BP(\bold S)|_\xi}(\d_1)$ has empty zero locus.

On the other hand, we have that $(N\otimes \cO_\xi(-\d_2+\d_1))^*\otimes \cO_\xi(\d_1)= N^*\otimes \cO_\xi(\d_2)$ is a quotient of $\hat{\bold M}_2^*\otimes \cO_\xi(\d_2)$, so in particular it is generically ample on $\xi$. Then Lemma~\ref{M1first} implies that the canonical section of $q^*(N\otimes \cO_\xi(-\d_2+\d_1))\otimes \cO_{\BP(\bold S)|_\xi}(\d_1)$ is nonempty. 
Contradiction.
\end{proof}

\end{document}